\documentclass[twoside,11pt]{article}

\usepackage[margin=2cm]{geometry}
\usepackage{amssymb,latexsym,amsmath,amsthm}
\usepackage[all]{xy}

\newtheorem{proposition}{Proposition}[section]
\newtheorem{theorem}[proposition]{Theorem}
\newtheorem{lemma}[proposition]{Lemma}

\newtheorem{definition}[proposition]{Definition}
\theoremstyle{remark}

\newcommand{\G}{\mathbb{G}}

\newcommand{\wap}{\operatorname{wap}}

\newcommand{\mc}{\mathcal}

\newcommand{\ip}[2]{\langle #1,#2 \rangle}

\newcommand{\vnten}{\overline\otimes}
\newcommand{\proten}{\widehat\otimes}

\newcommand{\id}{\operatorname{id}}
\newcommand{\scten}{\stackrel{\operatorname{sc}}{\otimes}}
\newcommand{\aone}{\Box}
\newcommand{\atwo}{\Diamond}
\newcommand{\M}{{\textsf{M}}}
\newcommand{\N}{{\textsf{N}}}
\newcommand{\Meas}{\mathfrak{M}}

\begin{document}

\title{Non-commutative separate continuity and weakly almost periodicity
for Hopf von Neumann algebras}
\author{Matthew Daws}
\maketitle

\begin{abstract}
For a compact Hausdorff space $X$, the space $SC(X\times X)$ of separately
continuous complex valued functions on $X$ can be viewed as a $C^*$-subalgebra
of $C(X)^{**}\vnten C(X)^{**}$, namely those elements which slice into $C(X)$.
The analogous definition for a non-commutative
$C^*$-algebra does not necessarily give an algebra, but we show that there is
always a greatest $C^*$-subalgebra.  This thus gives a non-commutative
notion of separate continuity.  The tools involved are multiplier algebras and
row/column spaces, familiar from the theory of Operator Spaces.  We make some
study of morphisms and inclusions.  There is a
tight connection between separate continuity and the theory of weakly almost
periodic functions on (semi)groups.  We use our non-commutative tools to show
that the collection
of weakly almost periodic elements of a Hopf von Neumann algebra, while itself
perhaps not a $C^*$-algebra, does always contain a greatest $C^*$-subalgebra.
This allows us to give a notion of non-commutative, or quantum, semitopological
semigroup, and to briefly develop a compactification theory in this context.


\end{abstract}

\section{Introduction}

A semigroup which carries a topology is \emph{semitopological} if the product
is separately continuous.  Natural examples arise when studying semigroup
compactifications of groups, or from semigroups of operators on reflexive
Banach spaces (indeed, these are linked, see \cite[Theorem~4.6]{meg}).
Indeed, for a locally compact group $G$, we say that
$f\in C_b(G)$ is \emph{weakly almost periodic} if the collection of (left,
or equivalently, right) translates of $f$ forms a relatively weakly compact
subset of $C_b(G)$.  The collection of all such functions, $\wap(G)$, forms
a $C^*$-subalgebra of $C_b(G)$ with character space $G^{\wap}$ say.  The
product on $G$ can be extended to $G^{\wap}$, and the resulting semigroup is
compact and semitopological.  Furthermore, $G^{\wap}$ is the maximal compact
semitopological semigroup to contain a dense homomorphic copy of $G$; see
\cite{bjm} and references therein.

We are interested in a ``quantum group'' approach to such questions.  A
compact semigroup $S$ canonically gives rise to a $C^*$-bialgebra by
considering $A=C(S)$ and the coproduct $\Delta:A\rightarrow A\otimes A
\cong C(S\times S)$ defined by $\Delta(f)(s,t) = f(st)$.  If $S$ is only
semitopological, then the natural codomain of $\Delta$ is now $SC(S\times S)$
the space of separately continuous functions.  Given the tight connection
between weakly almost periodic functions and separate continuity, it seems
likely that to study a notion of ``weakly almost periodicity'' for, say,
$C^*$-bialgebras, or Hopf von Neumann algebras, we will need a good notion of
``separate continuity'' for general $C^*$-algebras.  Indeed, for commutative
Hopf von Neumann algebras, similar connections were explored fruitfully in
\cite[Section~4]{dawswap}.
This paper develops a suitable non-commutative framework to attack this problem.

The most important class of Hopf von Neumann algebras are those arising
from locally compact quantum groups, \cite{kvvn}, as these completely
generalise the classical group algebras $L^1(G)$ and the ``dual picture'',
the Fourier algebra $A(G)$, \cite{eymard}.  For a locally compact
quantum group $\G$ we consider the Hopf von Neumann algebra
$(L^\infty(\G),\Delta)$, and use $\Delta$ to turn the predual $L^1(\G)$ into
a (completely contractive) Banach algebra.  There is a notion of a
\emph{weakly almost periodic} element of the dual of a Banach algebra (see
Section~\ref{sec:wap} below) and thus one can talk of $\wap(L^1(\G))$.
When applied to $L^1(G)$ this theory exactly recovers $\wap(G)$, see
\cite{ulger} for example.  For the Fourier algebra, this was suggested as early
as \cite{gran}; recent work for quantum groups can be found in 
\cite[Section~4]{runde2} and \cite{hnr}, for example.
However, outside of the commutative situation, to our knowledge there has
been little study in terms of compactifications (for different notions of
compactification, compare \cite{daws, salmi, soltan}).  In particular, it
is unknown, except in a few cases (see Section~\ref{sec:open} below)
if $\wap(L^1(\G))$ is a $C^*$-algebra.

Using our non-commutative notion of separate continuity, we show that
$\wap(L^1(\G))$ always contains a greatest $C^*$-subalgebra (that is,
a $C^*$-subalgebra containing all others) which we denote by
$\wap(L^\infty(\G),\Delta)$, to avoid confusion.  In fact,
$x\in \wap(L^\infty(\G),\Delta)$ if and only if $x\in\wap(L^1(\G))$ and
also $x^*x,xx^*\in\wap(L^1(\G))$.  (The reader should note that, using this
definition, it is unclear why $x,y\in\wap(L^\infty(\G),\Delta)$ implies
that $xy\in\wap(L^\infty(\G),\Delta)$; the equivalent properties established
earlier in the paper make this clear.)  We show that our notion of
separate continuity is stable under completely bounded maps, and this is used
to show that $\wap(L^\infty(\G),\Delta)$ is an $L^1(\G)$-submodule.

Using the notion of non-commutative separate continuity, we also make a
(tentative) definition of a quantum semitopological semigroup, and show that
$\wap(L^\infty(\G),\Delta)$ fits into this framework, and can actually be
interpreted as a ``compactification'' in this category.  Thus we obtain a
rather satisfactory theory, into which our recent notion of a $C^*$-Eberlein
algebra, \cite{dawsdas}, fits.  Whether we can extend other aspects of the
$C^*$-Eberlein algebra theory, for example invariant means and decompositions,
remains an open problem, see Section~\ref{sec:open}.

The paper is organised as follows: we make some preliminary remarks, mostly
to fix notation and terminology.  In Section~\ref{sec:one} we motivate
a tentative definition of ``separate continuity'' for an arbitrary
$C^*$-algebra, show that this doesn't in general yield an algebra, but then
give Theorem~\ref{thm:main} which establishes equivalent conditions on elements
which together form the maximal $C^*$-subalgebra.  This gives the notion of
$A\scten B$, the notation chosen as $C(X)\scten C(Y) \cong SC(X\times Y)$ for
compact Hausdorff spaces $X,Y$.  In Section~\ref{sec:for_vn_algs} we simplify
the theory in the case of von Neumann algebras.
In Section~\ref{sec:morphism} we study inclusions of
$C^*$-algebras (and establish a very simple slice map property) and also
show that $\scten$ is stable under completely bounded maps.  We then apply
this theory to Hopf von Neumann algebras in Section~\ref{sec:wap}.  For a locally
compact group, we have the choice as to work with $L^\infty(G)$ or perhaps
$C_0(G)$ (and these give the same notion of weakly almost periodic function).
Motivated by this, in Section~\ref{sec:cty} we look at $C^*$-bialgebras, and
show how to consistently use the Hopf von Neumann theory here.  We end with
some open questions.

\subsection{Acknowledgments}

The author is extremely grateful to Taka Ozawa who suggested to look at the
conditions in Theorem~\ref{thm:main}.  The work was initiated at the meeting
``Operator Methods in Harmonic Analysis'' at Queen's University Belfast,
organised by Ivan Todorov and Lyudmila Turowska and partially supported by
the London Mathematical Society.  The author thanks Piotr So{\l}tan and the referee
for helpful comments.

\section{Preliminaries}

For a Hilbert space $H$ we write the inner-product as $(\cdot|\cdot)$.
We write $\mc B(H)$ for the bounded linear maps on
$H$, and write $\mc B_0(H)$ for the ideal of compact operators.  Thus
$\mc B_0(H)$ is the closed span of the rank-one operators of the form
$\theta_{\xi,\eta}: \alpha \mapsto (\alpha|\eta)\xi$.  We denote by
$\omega_{\xi,\eta}$ the normal functional $\mc B(H)\rightarrow\mathbb C;
x\mapsto (x\xi|\eta)$.

We write $\otimes$ and $\vnten$ for tensor products of $C^*$-algebras and
von Neumann algebras, respectively.  For a von Neumann algebra $\M$ we
write $\M_*$ for its predual.
The multiplier algebra of a $C^*$-algebra $A$
is denoted by $M(A)$, and we use the notion of a non-degenerate
$*$-homomorphism $\theta:A\rightarrow M(B)$; we always have the strictly
continuous extension $\tilde\theta:M(A)\rightarrow M(B)$, see
\cite[Appendices]{mnw} for example.

We work with the standard theory of Operator Spaces, \cite{ER, Pisier}.
In particular, we write $\proten$ for the operator space projective tensor
product: the main result we use is that for von Neumann algebras $\M,\N$,
the predual of $\M\vnten \N$ is $\M_* \proten \N_*$.

While we mostly work with arbitrary Hopf von Neumann algebras (or, later,
$C^*$-bialgebras), the motivating examples come from the theory of locally
compact quantum groups, \cite{kvvn,kv}.  We follow the now standard notation
that for a locally compact quantum group $\G$, we write $L^\infty(\G)$ for
the von Neumann algebraic version, $L^1(\G)$ for its predual, and $C_0(\G)$
for the $C^*$-algebra version.

Finally, we use standard properties of weakly compact operators between
Banach spaces, see \cite[Section~5, Chapter~VI]{conway} for example:
if $T:E\rightarrow F$ is a bounded linear map then $T$ is weakly compact
if and only if $T^*:F^*\rightarrow E^*$ is weakly compact, if and only if
$T^{**}:E^{**} \rightarrow F^{**}$ maps into $F$.

Other terminology and theory will be introduced as needed.

\section{Motivation, and separate continuity}\label{sec:one}

We are ultimately interested in studying analogues of ``weakly almost periodic''
functions, in the context of locally compact quantum groups, or just Hopf von
Neumann (or $C^*$-) algebras.  When $G$ is a locally compact group, we consider
the group algebra $L^1(G)$ and turn $L^\infty(G)$ into an $L^1(G)$-bimodule in the
usual way.  The following are equivalent for a function $F\in L^\infty(G)$:
\begin{enumerate}
\item The orbit map $L^1(G)\rightarrow L^\infty(G), a\mapsto a\cdot F$ (or
equivalently $F\cdot a)$ is a weakly compact operator;
\item $F\in C_b(G)$ and the collection of left (or equivalently right) translates
of $F$ forms a relatively weakly compact subset of $F$.
\end{enumerate}
Compare \cite{ulger} for example.  As explained in the introduction,
the collection of such $F$ forms a unital $C^*$-subalgebra of $C_b(G)$, and the
character space is the compact semitopological semigroup $G^{\wap}$.

This gives some high-level motivation for looking at separately continuous
functions.  Furthermore, in \cite{dawswap} for example, consideration of spaces
of separately continuous functions proved to be a very useful technical tool.
Let $X$ be a compact Hausdorff space and denote by $SC(X\times X)$ the
space of separately continuous functions $X\times X\rightarrow\mathbb C$.
Following the clear presentation of \cite[Section~2]{runde} (ultimately using
work of Grothendieck) we have that for $f\in SC(X\times X)$ and $\mu,\lambda
\in \Meas(X)$, measures on $X$, if we define $(\mu\otimes\id)f, (\id\otimes\lambda)f
: X\rightarrow\mathbb C$ by
\[ ((\mu\otimes\id)f)(x) = \int_X f(y,x) \ d\mu(y), \qquad
((\id\otimes\lambda)f)(x) = \int_X f(x,y) \ d\lambda(y)
\qquad (x\in X), \]
then both $(\mu\otimes\id)f, (\id\otimes\lambda)f$ are in $C(X)$, and we have a
generalised Fubini-theorem,
\[ \ip{\lambda}{(\mu\otimes\id)f} = \ip{\mu}{(\id\otimes\lambda)f}. \]
Let us write $\mc B_{w^*}(\Meas(X)\times \Meas(X),\mathbb C)$ for the space of separately
weak$^*$-continuous bilinear maps $\Meas(X)\times \Meas(X)\rightarrow \mathbb C$.
Thus we have shown that $f\in SC(X\times X)$ induces $\Phi_f$, say, in 
$\mc B_{w^*}(\Meas(X)\times \Meas(X),\mathbb C)$.  Furthermore, this establishes
an isometric isomorphism between $SC(X\times X)$ and
$\mc B_{w^*}(\Meas(X) \times \Meas(X),\mathbb C)$, see \cite[Proposition~2.5]{runde}.

Linearising the bilinear map, we get the projective tensor product
$\Meas(X) \proten \Meas(X)$ which is the predual of the commutative von Neumann algebra
$C(X)^{**} \vnten C(X)^{**}$.  Then the separately weak$^*$-continuous members of
$(\Meas(X) \proten \Meas(X))^* \cong C(X)^{**} \vnten C(X)^{**}$ are precisely those
$x\in C(X)^{**} \vnten C(X)^{**}$ which slice into $C(X)$.  Hence we obtain that
\[ SC(X\times X) \cong \{ x \in C(X)^{**} \vnten C(X)^{**} :
(\mu\otimes\id)x, (\id\otimes\mu)x \in C(X) \ (\mu\in \Meas(X)) \}. \]

This motivates, for a unital, but maybe noncommutative, $C^*$-algebra $A$, the
definition that
\[ SC(A\times A) = \{ x\in A^{**}\vnten A^{**} : (\mu\otimes\id)x,
(\id\otimes\mu)x \in A \ (\mu\in A^*) \}. \]
Unfortunately, this need not be an algebra in general, as hinted at in
\cite[Section~3]{dawsdas}.  For example, let $H$ be an infinite-dimensional
Hilbert space and set $A=\mc B_0(H)$ the compact operators on $H$.  Then
$A^{**}\cong \mc B(H)$ and $A^{**}\vnten A^{**}\cong \mc B(H\times H)$.  Consider
$\Sigma \in \mc B(H\times H)$ the swap map.  Then for $\omega_{\xi,\eta} \in
\mc B(H)_*$ and $\alpha,\beta\in H$,
\[ \big( (\omega_{\xi,\eta}\otimes\id)(\Sigma) \alpha \big| \beta \big)
= \big( \alpha\otimes\xi \big| \eta\otimes \beta \big)
= (\alpha|\eta)(\xi|\beta) = \big( \theta_{\xi,\eta} \alpha \big| \beta \big), \]
so $(\omega_{\xi,\eta}\otimes\id)(\Sigma) = \theta_{\xi,\eta}\in \mc B_0(H)$
and similarly $(\id\otimes\omega_{\xi,\eta})(\Sigma) = \theta_{\xi,\eta}\in
\mc B_0(H)$.  By linearity and continuity, we conclude that
$\Sigma\in SC(A\times A)$.  However, of
course $\Sigma^2 = 1_{H\otimes H}$ and slices of $\Sigma^2$ give all of
$\mathbb C 1_H \not\subseteq A$.

\subsection{Non-commutative separate continuity}

The following idea originates from a suggestion of Taka Ozawa.
For $C^*$-algebras $A,B$ we define, as above
\[ SC(A\times B) = \{ x\in A^{**}\vnten B^{**} : (\mu\otimes\id)x\in B,
(\id\otimes\lambda)x\in A \ (\mu\in A^*, \lambda\in B^*) \}. \]
Our aim is to find a maximal $C^*$-subalgebra of $SC(A\times B)$, this
will actually turn out to the ``maximum'' $C^*$-subalgebra.

We first recall some definitions from \cite[Section~0]{bs}.  For a $C^*$-algebra
$A$ and for $J$ a (closed, two-sided) ideal
in $A$, let $M(A;J) = \{ x\in M(A) : xA+Ax\subseteq J \}$, which is a
$C^*$-subalgebra of $M(A)$, and by restriction to $J$, is isomorphic to a
$C^*$-subalgebra of $M(J)$.  Let $A^1$ be the conditional unitisation of $A$, so
for an auxiliary $C^*$-algebra $B$, we can consider $M(A^1\otimes B;A\otimes B)$.
It is easy to see that this algebra consists of those $x\in M(A^1\otimes B)$
such that $x(1\otimes b),(1\otimes b)x\in A\otimes B$ for all $b\in B$.

In the following theorem we shall consider a $C^*$-algebra $B \subseteq\mc B(K)$
for a Hilbert space $K$.  Using a fixed orthonormal basis $(e_j)_{j\in J}$ of $K$
we can view members of $\mc B(K)$ as being $J\times J$ matrices, say
$\mc B(K) \subseteq \mathbb M_J$.  Given another $C^*$-algebra $A$ we can extend
this identification to view members of $A^{**}\vnten\mc B(K)$ as being $A^{**}$-valued
matrices, say $x\in A^{**}\vnten\mc B(K)$ corresponds to
$(x_{ij}) \in \mathbb M_J(A^{**})$.

\begin{theorem}\label{thm:main}
Let $A,B$ be $C^*$-algebras represented on Hilbert spaces $H,K$ such that the
induced maps $\mc B(H)_*\rightarrow A^*$ and $\mc B(K)_*\rightarrow B^*$ are
both onto (for example, these could be the universal representations).
We may then regard $A^{**}\vnten B^{**}$ as a subalgebra of $\mc B(H\otimes K)$.
For $x\in A^{**}\vnten B^{**}$ the following are equivalent:
\begin{enumerate}
\item\label{thm:main:one} $x, x^*x, xx^*\in SC(A\times B)$;
\item\label{thm:main:two} $x \in M(A^1\otimes \mc B_0(K);A\otimes \mc B_0(K))
\cap M(\mc B_0(H) \otimes B^1;\mc B_0(H) \otimes B)$;
\item\label{thm:main:three} we embed $A^{**}\vnten B^{**}$ into $A^{**}\vnten
\mc B(K)$ and view $x = (x_{ij}) \in \mathbb M_J(A^{**})$ as above.  We require that
each $x_{ij}\in A$, and that both $\sum_j x_{ij} x_{ij}^*$ and $\sum_j x_{ji}^* x_{ji}$
are norm convergent sums, for each $i$.  Similarly with the roles of $A$ and $B$ swapped.
\end{enumerate}
\end{theorem}

We now proceed to prove Theorem~\ref{thm:main} with an aim to prove a little more
than strictly necessary.  Throughout the rest of this section, $A,B$ will be
$C^*$-algebras.  Let $\pi_A:A\rightarrow\mc B(H), \pi_B:B\rightarrow\mc B(K)$ be
arbitrary, non-degenerate, $*$-homomorphisms with normal extensions $\tilde\pi_A:
A^{**}\rightarrow\mc B(H), \tilde\pi_B:B^{**}\rightarrow\mc B(K)$, see for
example \cite[Section~2, Chapter~III]{tak}.
Notice that under the identification of $x\in A^{**}\vnten\mc B(K)$ with
$(x_{ij}) \in \mathbb M_J(A^{**})$ we have that $x_{ij} = (\id\otimes\omega_{e_j,e_i})(x)$
(observe the order of indices).  In particular, if
$x = (\tilde\pi_A\otimes\tilde\pi_B)(y)$ for some
$y\in SC(A\times B)$ then $x_{ij}\in \pi_A(A)$ for all $i,j$.

\begin{proposition}\label{prop:2nd_two_equivs}
Let $\pi_A,\pi_B$ be as above, and let $x\in \mc B(H\otimes K)$.
The following are equivalent:
\begin{enumerate}
\item\label{prop:2nd_two_equivs:one}
$x\in M(\pi_A(A^1)\otimes\mc B_0(K); \pi_A(A)\otimes\mc B_0(K))$;
\item\label{prop:2nd_two_equivs:two}
with $x$ identified with $(x_{ij})$ we have that $x_{ij}\in \pi_A(A)$ for all
$i,j$, and both $\sum_j x_{ij} x_{ij}^*$ and $\sum_j x_{ji}^* x_{ji}$ are norm
convergent sums in $\mc B(H)$, and hence in $\pi_A(A)$, for each $i$.
\end{enumerate}
\end{proposition}
\begin{proof}
Firstly, observe that (\ref{prop:2nd_two_equivs:one}) is equivalent to
$x(1\otimes\theta), (1\otimes\theta)x \in \pi_A(A)\otimes\mc B_0(K)$ for all
$\theta\in\mc B_0(K)$.  Let us consider the case of $x(1\otimes\theta)$.
By linearity and continuity, we need only check that $x(1\otimes\theta_{e_k,e_l})\in \pi_A(A)\otimes\mc B_0(K)$ for all $k,l\in J$.

Fix $k,l\in J$ and let $z=x(1\otimes \theta_{e_k,e_l})$.
Considering $z$ as a matrix in $\mathbb M_J(\mc B(H))$ we have that
\[ z_{ij} = (\id\otimes\omega_{e_j,e_i})(x(1\otimes\theta_{e_k,e_l}))
= \delta_{l,j}(\id\otimes\omega_{e_k,e_i})(x)
= \delta_{l,j} x_{ik} \in \pi_A(A). \]
Thus the matrix of $z$ actually consists of one non-zero column.

Suppose that (\ref{prop:2nd_two_equivs:two}) holds, let $J_0\subseteq J$ be a
finite subset, and define
\[ w_{ij} = \begin{cases} z_{ij} &: i\in J_0, \\
0 &:\text{otherwise} \end{cases} \]
Thus $w=(w_{ij})$ is a finitely supported matrix and $w_{ij}\in\pi_A(A)$ for
all $i,j$, so clearly $w\in \pi_A(A)\otimes\mc B_0(K)$.  As both $w$ and $z$
are just single columns, we immediately see that
\[ \|z-w\| = \Big\| \sum_{i\in J} (z_{il}-w_{il})^*(z_{il}-w_{il})\Big\|^{1/2}
= \Big\| \sum_{i\not\in J_0} z_{il}^*z_{il}\Big\|^{1/2}
= \Big\| \sum_{i\not\in J_0} x_{ik}^*x_{ik}\Big\|^{1/2}. \]
By assumption, $\sum_i x_{ik}^*x_{ik}$ converges in norm, and so for any
$\epsilon>0$ there is $J_0$ finite with $\|z-w\| \leq \epsilon$.
Thus $z \in \pi_A(A)\otimes\mc B_0(K)$.

Conversely, suppose (\ref{prop:2nd_two_equivs:one}) holds, and let $z$ be
as before, now known to be a member of $\pi_A(A)\otimes\mc B_0(K)$.
As $\pi_A(A)
\otimes\mc B_0(K) \subseteq \mathbb M_J(\mc B(H))$ is the norm closure of
finite matrices with entries in $\pi_A(A)$, for each $\epsilon>0$ there is
such a finite matrix $w=(w_{ij})$ with $\|z-w\|\leq\epsilon$.  As the operation
of projecting onto the $l$th column is (completely) contractive, we may suppose
without loss of generality that $w_{ij} = \delta_{l,j} w_{ij}$ for all $i,j$.
Thus $(z-w)$ has matrix consisting of just a non-zero column, and so
\[ \epsilon \geq \|z-w\| = \Big\| \sum_i (z_{il}-w_{il})^*(z_{il}-w_{il})
\Big\|^{1/2} = \Big\| \sum_i (x_{ik}-w_{il})^*(x_{ik}-w_{il}) \Big\|^{1/2}. \]
If $w_{il} = 0$ for $i\not\in J_0$ with $J_0$ finite, then in particular
\[ \Big\| \sum_{i\not\in J_0} x_{ik}^*x_{ik} \Big\|^{1/2} \leq \epsilon, \]
as required.

We have hence shown that $x(1\otimes\theta)\in \pi_A(A)\otimes\mc B_0(K)$ for
all $\theta\in \mc B_0(K)$, if and only if $\sum_j x_{ji}^*x_{ji}$ is norm
convergent for each $i$.
The other case follows similarly, or from simply replacing $x$ by $x^*$.
\end{proof}

\begin{proof}[Proof of Theorem~\ref{thm:main}]
That (\ref{thm:main:two}) and (\ref{thm:main:three}) are equivalent
follows from Proposition~\ref{prop:2nd_two_equivs}.

Suppose that (\ref{thm:main:two}) holds.  Let $\xi,\eta\in K$,
let $i,j\in J$, and set $y = x(1\otimes\theta_{e_i,\eta}), z =
x(1\otimes\theta_{e_j,\xi})$.  Then
\[ z^*y = (1\otimes\theta_{\xi,e_j}) x^*x (1\otimes\theta_{e_i,\eta})
= (\id\otimes\omega_{e_i,e_j})(x^*x) \otimes \theta_{\xi,\eta}. \]
By assumption, $y,z\in A\otimes\mc B_0(K)$, and so
$(\id\otimes\omega_{e_i,e_j})(x^*x) \in A$.  By linearity and continuity,
and using that $\mc B(K)_*\rightarrow B^*$ is onto, it follows that
$(\id\otimes\mu)(x^*x)\in A$ for all $\mu\in B^*$.  Similarly,
$(\id\otimes\mu)(xx^*)\in A$ for all $\mu\in B^*$.  Furthermore,
$(\id\otimes\omega_{\xi,e_j})(y) = (\xi|\eta) (\id\otimes\omega_{e_i,e_j})(x)
\in A$ and so by a suitable choice of $\xi,\eta$, and again by linearity
and continuity, we conclude that $(\id\otimes\mu)(x) \in A$ for all
$\mu\in B^*$.  Repeating the argument with
the roles of $A$ and $B$ swapped shows that (\ref{thm:main:one}) holds.

Conversely, suppose that (\ref{thm:main:one}) holds.  By the discussion
above, the matrix $(x_{ij})$ does consist of elements of $A$.
The matrix representation
of $x^*x$ is $(x^*x)_{ij} = \sum_k x_{ki}^* x_{kj}$, the sum converging strongly
in $\mc B(H)$, for example.  By assumption, $(x^*x)_{ij}\in A$.
For each positive $\mu\in A^*$ choose $\omega\in\mc B(H)_*$ with $\omega|_A = \mu$.
Writing $J_0\subset\subset J$ to indicate that $J_0$ is a finite subset of $J$,
we see that
\begin{align*}
\sup_{J_0\subset\subset J} \ip{\mu}{\sum_{k\in J_0} x_{ki}^* x_{kj}}
&= \sup_{J_0\subset\subset J} \ip{\sum_{k\in J_0} x_{ki}^* x_{kj}}{\omega}
= \lim_{J_0\subset\subset J}\ip{\sum_{k\in J_0} x_{ki}^* x_{kj}}{\omega} \\
&= \sum_{k\in J} \ip{x_{ki}^* x_{kj}}{\omega}
= \ip{(x^*x)_{ij}}{\omega}
= \ip{\mu}{(x^*x)_{ij}}.
\end{align*}
So we have an increasing net in $A^+$ which converges, against elements of
the unit ball of $A^*_+$, to an element of $A^+$.
By an application of Dini's Theorem, compare \cite[Lemma~A.3]{kv}, it follows
that $\sum_{k} x_{ki}^* x_{kj} = (x^*x)_{ij}$ with convergence in norm in $A$.
Applying a similar argument to $xx^*$, and then swapping the roles of $A$
and $B$, shows that (\ref{thm:main:three}) holds.
\end{proof}

Let us now define the object which we shall study for the rest of the paper.

\begin{definition}
For $C^*$-algebras $A$ and $B$ define
\[ A\scten B = \{ x\in SC(A\times B) :
x^*x,xx^*\in SC(A\times B)\}. \]
\end{definition}

\begin{theorem}\label{thm:main_app}
For $C^*$-algebras $A,B$ we have that $A\scten B$ is a $C^*$-subalgebra of
$A^{**}\vnten B^{**}$ and every $*$-algebra contained in $SC(A\times B)$ is
contained in $A\scten B$.
\end{theorem}
\begin{proof}
The first claim follows immediately from the equivalence of (\ref{thm:main:one})
and (\ref{thm:main:two}) above, as (\ref{thm:main:two}) is stated in terms of
the intersection of two $C^*$-algebras.  The second claim is immediate.
\end{proof}

\subsection{For von Neumann algebras}\label{sec:for_vn_algs}

We now aim to apply this construction to von Neumann algebras $\M,\N$.
By definition, this would involve working in $\M^{**}\vnten \N^{**}$, but the
constructions in this section allow us to work with $\M\vnten \N$ instead.
When $\M,\N$ are commutative, similar (but less general) ideas are explored in
\cite[Section~4]{dawswap}.

Fix von Neumann algebras $\M,\N$ with preduals $\M_*,\N_*$.  Consider the canonical
map from a Banach space to its bidual $\kappa = \kappa_{\M_*}:\M_*\rightarrow
\M^*$.  Then $\kappa^* : \M^{**}\rightarrow \M$ is a $*$-homomorphism, normal
by construction.  In fact, $\kappa(\M_*)$ is 1-complemented in $\M^*$,
see \cite[Section~2, Chapter~III]{tak}.  Thus $\kappa^* \otimes \kappa^*$ is
a normal $*$-homomorphism $\M^{**}\vnten \N^{**} \rightarrow \M\vnten \N$.  Let
$\theta_{sc}$ be the restriction of this map to $SC(\M\times \N)$, so $\theta_{sc}$
further restricted to $\M\scten \N$ is a $*$-homomorphism, separately
weak$^*$-continuous.

By analogy with the Banach algebra situation (see Section~\ref{sec:wap} below,
or, if one prefers, a completely bounded analogue of Arens's original work,
\cite{arens})
define $\wap(\M\vnten \N)$ to be those $x\in \M\vnten \N$ such that the maps
\[ L_x : \M_*\rightarrow \N, \omega\mapsto (\omega\otimes\id)(x), \qquad
R_x:\N_*\rightarrow \M, \omega\mapsto (\id\otimes\omega)(x) \]
are weakly compact.  Notice that $L_x^* \circ \kappa_{\N_*} = R_x$ and
$R_x^*\circ \kappa_{\M_*} = L_x$, and so $L_x$ is weakly compact if and only if
$R_x$ is.

\begin{lemma}\label{lem:into_wap}
Let $x\in SC(\M\times \N)$ and set $y=\theta_{sc}(x)$.  Then $y\in\wap(\M\vnten \N)$.
\end{lemma}
\begin{proof}
For $\omega\in \N_*$ and $\tau\in \M_*$ we have that
\[ \ip{R_y(\omega)}{\tau} = \ip{x}{\kappa(\tau)\otimes\kappa(\omega)}
= \ip{(\id\otimes\kappa(\omega))(x)}{\kappa(\tau)}
= \ip{(\id\otimes\kappa(\omega))(x)}{\tau}, \]
as by assumption, $(\id\otimes\kappa(\omega))(x) \in \M \subseteq \M^{**}$.

Let $(\omega_\alpha)$ be a bounded net in $N_*$ and by moving to a subnet if
necessary, suppose that $\omega_\alpha\rightarrow \mu\in N^*$ weak$^*$ in $N^*$.
Then for $\lambda\in \M^*$,
\[ \lim_\alpha \ip{\lambda}{R_y(\omega_\alpha)}
= \lim_\alpha \ip{\lambda}{(\id\otimes\kappa(\omega_\alpha))(x)}
= \lim_\alpha \ip{x}{\lambda \otimes \kappa(\omega_\alpha)}
= \lim_\alpha \ip{(\lambda\otimes\id)(x)}{\kappa(\omega_\alpha)}. \]
As $(\lambda\otimes\id)(x)\in N \subseteq N^{**}$ this limit is equal to
\[ \ip{(\lambda\otimes\id)(x)}{\mu} = \ip{\lambda}{(\id\otimes\mu)(x)}, \]
where $(\id\otimes\mu)(x)\in \M$.  Thus $R_y(\omega_\alpha) \rightarrow
(\id\otimes\mu)(x)\in \M$ weakly.  This establishes that $R_y$ is weakly compact,
as required.
\end{proof}

We will now proceed to show that the map $\theta_{sc}:SC(\M\times \N) \rightarrow
\wap(\M\vnten \N)$ is actually a bijection.  Firstly we show it is onto, for which
a further idea of Arens is required; we follow the notation of
\cite[Section~3]{dawsdis}, adapted to the von Neumann algebra situation.
Given $x\in \M\vnten \N$, for $\mu\in \M^*$ we define $(\mu\otimes\id)(x)\in \N$ by
\[ \ip{(\mu\otimes\id)(x)}{\omega} = \ip{\mu}{(\id\otimes\omega)(x)}
\qquad (\omega\in \N_*). \]
Similarly define $(\id\otimes\lambda)(x)\in \M$ for $\lambda\in \N^*$
Then we define two completely contractive maps
$(\M^{**}\vnten \N^{**})_* = \M^*\proten \N^*
\rightarrow (\M_*\proten \N_*)^{**} = (\M\vnten \N)^*$ by
\[ \mu\otimes\lambda \mapsto \mu \otimes_\aone \lambda, \quad
\mu\otimes_\atwo\lambda, \]
and extending by linearity and continuity, where we define
\[ \ip{\mu \otimes_\aone \lambda}{x} = \ip{\mu}{(\id\otimes\lambda)(x)},
\qquad \ip{\mu \otimes_\atwo \lambda}{x} = \ip{\lambda}{(\mu\otimes\id)(x)}. \]
(We note that the definition of $\otimes_\atwo$ in \cite[Page~16]{dawsdis}
is wrong, or at least inconsistent; one should swap $\Phi,\Psi$ in the
formula on page~16.)

The following again goes back to Arens, but we include a proof for reference
and motivation.

\begin{lemma}\label{lem:wap_char}
We have that $x\in\wap(\M\vnten \N)$ if and only if $\ip{\mu\otimes_\aone\lambda}{x}
= \ip{\mu\otimes_\atwo\lambda}{x}$ for all $\mu,\lambda$.
\end{lemma}
\begin{proof}
To show ``if'', let $(\omega_\alpha)$ be a bounded net in $\M_*$ converging weak$^*$ to
$\mu\in \M^*$.  For $\lambda\in \N^*$ we have that
\[ \lim_\alpha \ip{\lambda}{L_x(\omega_\alpha)}
= \lim_\alpha \ip{(\id\otimes\lambda)(x)}{\omega_\alpha}
= \ip{\mu\otimes_\aone\lambda}{x}
= \ip{\mu\otimes_\atwo\lambda}{x}
= \ip{\lambda}{(\mu\otimes\id)(x)}, \]
and so $L_x(\omega_\alpha) \rightarrow (\mu\otimes\id)(x)\in \N$ weakly.
As in the proof of Lemma~\ref{lem:into_wap}, it follows that
$x \in \wap(\M\vnten \N)$.

Conversely, let $(\tau_\beta)$ in $\N_*$ converge weak$^*$ to $\lambda\in \N^*$,
and let $(\omega_\alpha)$ as before.  Assuming that $L_x$ is weakly compact,
we may assume that $L_x(\omega_\alpha)\rightarrow x_0\in \M$, say, weakly.  Then
\begin{align*} \ip{\mu\otimes_\aone\lambda}{x}
&= \lim_\alpha \ip{(\id\otimes\lambda)(x)}{\omega_\alpha}
= \lim_\alpha \ip{\lambda}{L_x(\omega_\alpha)}
= \ip{\lambda}{x_0} = \lim_\beta \ip{x_0}{\tau_\beta}
= \lim_\beta \lim_\alpha \ip{L_x(\omega_\alpha)}{\tau_\beta} \\
&= \lim_\beta \lim_\alpha \ip{R_x(\tau_\beta)}{\omega_\alpha}
= \lim_\beta \ip{\mu}{R_x(\tau_\beta)}
= \ip{\lambda}{(\mu\otimes\id)(x)}
= \ip{\mu\otimes_\atwo\lambda}{x}, \end{align*}
as required.
\end{proof}

\begin{proposition}\label{prop:wap_iso_sc}
Let $x\in\wap(\M\vnten \N)$, and define $y\in \M^{**}\vnten \N^{**} = (\M^*\proten \N^*)^*$
by $\ip{y}{\mu\otimes\lambda} = \ip{\mu\otimes_\aone\lambda}{x}
= \ip{\mu\otimes_\atwo\lambda}{x}$.
Then $y\in SC(\M\times \N)$; indeed, $(\mu\otimes\id)(y) = (\mu\otimes\id)(x)\in \N$
and $(\id\otimes\lambda)(y) = (\id\otimes\lambda)(x)\in \M$.

Conversely, if $y\in SC(\M\times \N)$ and we set $x=\theta_{sc}(y)$ then
$(\mu\otimes\id)(x) = (\mu\otimes\id)(y)\in \N$ and
$(\id\otimes\lambda)(x) = (\id\otimes\lambda)(y)\in \M$.
As such, the map $\theta_{sc} : SC(\M\times \N) \rightarrow \wap(\M\vnten \N)$
is an isomorphism.
\end{proposition}
\begin{proof}
The first claim follows from a simple calculation.
Now let $y\in SC(\M\times \N)$ and set $x=\theta_{sc}(y)$.
Consider the biadjoint $R_x^{**} : \N^*\rightarrow \M^{**}$, which satisfies that
if $\mu\in \N^*$ is the weak$^*$-limit of $(\omega_\alpha)\subseteq \N_*$ then
$R_x^{**}(\mu)$ is the weak$^*$-limit, in $\M^{**}$, of the net $(R_x(\omega_\alpha))$.
As $R_x$ is weakly compact, this actually converges weakly in $\M$, and the proof
of Lemma~\ref{lem:wap_char} shows that $R_x^{**}(\mu) = (\mu\otimes\id)(x)$.
However, from the proof of Lemma~\ref{lem:into_wap}, this net converges to
$(\id\otimes\mu)(y) \in \M$, and so the result follows.
\end{proof}

\begin{theorem}\label{thm:sc_into_wap}
The map $\theta_{sc}: \M\scten \N \rightarrow \wap(\M\vnten \N)$ is an injective
$*$-homomorphism, separately normal, and has image those $x\in \wap(\M\vnten \N)$
such that also $x^*x, xx^* \in \wap(\M\vnten \N)$.
\end{theorem}
\begin{proof}
It follows from Proposition~\ref{prop:wap_iso_sc} that $\theta_{sc}$ is injective,
and it is separately normal by construction.  If $y\in \M\scten \N$ then also
$y^*y,yy^*\in \M\scten \N$.  Thus if
$x=\theta_{sc}(y)$ then $x^*x = \theta_{sc}(y^*y) \in \wap(\M\vnten \N)$, and
similarly $xx^* \in \wap(\M\vnten \N)$.  Conversely, if $x\in\wap(\M\vnten \N)$
with $x^*x, xx^* \in \wap(\M\vnten \N)$ then let $y\in SC(\M\times \N)$ with
$\theta_{sc}(y) = x$.  There is $z\in SC(\M\times \N)$ with $\theta_{sc}(z) =
x^*x$.  As $\theta_{sc}$ is the restriction of $\kappa^*\otimes\kappa^*$, which
is a $*$-homomorphism to $\M\vnten \N$,
\[ (\kappa^*\otimes\kappa^*)(z) = x^*x = (\kappa^*\otimes\kappa^*)(y^*y), \]
and so $z = y^*y$.  Thus $y^*y$, and similarly $yy^*$, are members of $SC(\M\times \N)$,
and so by definition, $y\in \M\scten \N$.
\end{proof}

Hence $\M\scten \N$ is the maximum $C^*$-subalgebra of $\wap(\M\vnten \N)$.
However, again the space $\wap(\M\vnten \N)$ need not be an algebra in general.
An example of $\M$ and $t\in\wap(\M\vnten \M)$ such that $t^*t\not\in
\wap(\M\vnten\M)$ may be constructed as follows.  Let $\M=\mc B(H)$ for a
separable, infinite-dimensional Hilbert space $H$.  It is easy to find a positive
$x\in\mc B(H\otimes H)= \M\vnten\M$ such that the map $\M_*\rightarrow\M; \omega\mapsto
(\omega\otimes\id)(x)$ is not weakly compact.  Now let $u:H\rightarrow H\otimes H$
be a unitary, and fix a unit vector $\xi_0\in H$.  Then $t\in\mc B(H\otimes H)$
defined by $t(\xi) = u^*x^{1/2}(\xi) \otimes \xi_0$ has the required property.
This follows as $t^*t = x \not\in \wap(\M\vnten \M)$, while one can show that the
maps $\M_*\rightarrow\M; \omega\mapsto (\omega\otimes\id)(t)$ and
$\omega\mapsto (\id\otimes\omega)(t)$ both factor through a Hilbert space, and
so are weakly compact.

\section{Morphisms and inclusions}\label{sec:morphism}

In this section we study stability properties of $\scten$.  We start by considering
inclusions.  Let $A$ be a $C^*$-algebra and let $A_0\subseteq A$ be a $C^*$-subalgebra.
Then the inclusion $\iota:A_0\rightarrow A$
induces the inclusion $\iota^{**}:A_0^{**} \rightarrow A^{**}$ which is a
normal $*$-homomorphism.
Indeed, if we identify $A_0$ with a subalgebra of $A$, then the restriction map
$A^*\rightarrow A_0^*$ is a quotient map, with kernel $A_0^\perp = \{ \mu\in A^* :
\ip{\mu}{a}=0 \ (a\in A_0) \}$.  Then $A_0^{**} = (A_0^*)^* \cong
(A^*/A_0^\perp)^* \cong A_0^{\perp\perp}$.
If also $B_0\subseteq B$ is an inclusion of $C^*$-algebras, then we have the
chain of isometric inclusions
\[ A_0 \scten B_0 \subseteq SC(A_0\times B_0) \subseteq A_0^{**}\vnten B_0^{**}
\subseteq A^{**} \vnten B^{**}. \]
The following result gives a simple ``slice map''
criteria to determine membership of $A_0\scten B_0$.

\begin{theorem}\label{thm:inclusion}
For $x\in A^{**}\vnten B^{**}$ the following are equivalent:
\begin{enumerate}
\item\label{thm:inclusion:one}
$x$ is in (the image of) $A_0 \scten B_0$;
\item\label{thm:inclusion:three}
$x\in A\scten B$ and $x$ is in (the image of) $SC(A_0\times B_0)$
(that is, $(\mu\otimes\id)(x)\in B_0, (\id\otimes\lambda)(x)\in A_0$
for $\mu\in A^*,\lambda\in B^*$).
\end{enumerate}
\end{theorem}
\begin{proof}
(\ref{thm:inclusion:one})$\implies$(\ref{thm:inclusion:three}):
For $x\in A_0\scten B_0$ and for $\mu\in B^*$, letting $\mu_0\in B_0^*$ be the
restriction, we have that $(\id\otimes\mu)(x) = (\id\otimes\mu_0)(x)
\in A_0 \subseteq A$.  Similarly $(\id\otimes\mu)(x^*x)\in A$ and
$(\id\otimes\mu)(xx^*)\in A$.  Analogously, right slices of $x,x^*x,xx^*$ are
in $B$, and so $x\in A\scten B$.  Clearly also $x\in SC(A_0\times B_0)$.

(\ref{thm:inclusion:three})$\implies$(\ref{thm:inclusion:one}):
Let $B\subseteq\mc B(K)$ be the universal representation and let $K\cong\ell^2(J)$,
so we can regard $A^{**}\vnten B^{**} \subseteq\mathbb M_J(A^{**})$ again.
From Theorem~\ref{thm:main} we know that if $x=(x_{ij})$ then the sums
$\sum_j x_{ij}x_{ij}^*$ and $\sum_j x_{ji}^*x_{ji}$ converge in norm.
That $x\in SC(A_0\times B_0)$ tells us that each $x_{ij} \in A_0$, and hence
that the sums actually converge in $A_0$.  We apply similar
arguments with the roles and $A$ and $B$ swapped, and then by
Theorem~\ref{thm:main} again we conclude that $x\in A_0\scten B_0$.
\end{proof}

Let us make the following simple remark.  As $A\subseteq A^{**}$ and
$B\subseteq B^{**}$ we have the inclusion $A\scten B\subseteq A^{**}\scten B^{**}$.
As $A^{**}, B^{**}$ are von Neumann algebras, we can apply
Theorem~\ref{thm:sc_into_wap} to identify $A^{**}\scten B^{**}$ as a subalgebra
of $A^{**}\vnten B^{**}$.  The composition gives an inclusion $A\scten B
\rightarrow A^{**}\vnten B^{**}$, and this is nothing but the canonical
inclusion.

\subsection{Completely bounded maps}

In this section we show that $\scten$ is stable under completely bounded maps.

\begin{theorem}\label{thm:morphisms}
Let $A,B,C$ be $C^*$-algebras and let $\phi:A\rightarrow B$ be a completely
bounded map.  Then $\phi^{**}\otimes id:A^{**}\vnten C^{**} \rightarrow
B^{**}\vnten C^{**}$ is completely bounded, and restricts to a map
$\phi\scten\id: A\scten C \rightarrow B\scten C$.
\end{theorem}

From the theory of operator spaces, we know that $\phi^*:B^*\rightarrow A^*$
is completely bounded and hence also $\phi^*\otimes\id:
B^*\proten C^*\rightarrow A^*\proten C^*$ is completely bounded, all with equal
norms.  Hence we do obtain $\phi^{**}\otimes\id:A^{**}\vnten C^{**} \rightarrow
B^{**}\vnten C^{**}$.

Alternatively, from the structure theory of completely
bounded maps (see \cite[Theorem~8.4]{paulsen} for example) there is a Hilbert
space $L$,
a non-degenerate $*$-representation $\pi:A\rightarrow\mc B(L)$ and
$T,S\in\mc B(H',L)$ with $\phi(a) = T^*\pi(a)S \in B\subseteq\mc B(H')$.
Here we may, and will, choose $B\subseteq\mc B(H')$ to be the universal
representation.  Let $\tilde\pi:A^{**}\rightarrow\mc B(L)$ be the normal
extension.  Then a simple calculation shows that $\phi^{**}:A^{**}\rightarrow
B^{**}$ has the form $\phi^{**}(x) = T^*\tilde\pi(x)S \in B^{**}\cong B''
\subseteq\mc B(H')$.  Let $A\subseteq\mc B(H)$ be the universal representation,
so again $A^{**} \cong A'' \subseteq\mc B(H)$.
By the structure theory for normal $*$-representations,
\cite[Theorem~5.5, Chapter~IV]{tak}, there is an auxiliary Hilbert space $L'$,
a projection $p\in A'\vnten\mc B(L') \subseteq \mc B(H\otimes L')$
and a unitary $U:p(H\otimes L')\rightarrow L$ such that
\[ \tilde\pi(x) = Up(x\otimes 1_{L'})U^* \qquad (x\in A^{**}). \]

Consequently, by enlarging the original $L$ if necessary, we may actually assume
that $L = H\otimes L'$, and that $\pi(a)=a\otimes 1$.  Thus $T,S\in\mc B(H',
H\otimes L')$ and
\[ \phi(a) = T^*(a\otimes 1)S, \qquad \phi^{**}(x) = T^*(x\otimes 1)S
\qquad (a\in A, x\in A^{**}). \]
We remark that, even in this setting, we can always choose $S,T$ with
$\|S\| \|T\| = \|\phi\|_{cb}$.
If also $C\subseteq\mc B(K)$ is the universal representation,
then $A^{**}\vnten C^{**}\subseteq \mc B(H)\vnten \mc B(K) = \mc B(H\otimes K)$
and similarly $B^{**}\vnten C^{**}\subseteq \mc B(H'\otimes K)$.  Then
\[ (\phi^{**}\otimes\id)(x) = (T\otimes 1_K)^* x_{13} (S\otimes 1_K)
\qquad (x\in A^{**}\vnten C^{**}). \]
Here we use the ``leg numbering notation'', so $x_{13}$ is $x\in\mc B(H\otimes K)$
acting on the 1st and 3rd components of $H\otimes L'\otimes K$.

\begin{proof}[Proof of Theorem~\ref{thm:morphisms}]
Let $x\in A\scten C \subseteq A^{**}\vnten C^{**}$.  We shall verify condition
(\ref{thm:main:two}) of Theorem~\ref{thm:main} for $(\phi^{**}\otimes\id)(x)$.
For $\theta\in\mc B_0(K)$, from the above discussion,
\begin{align*} (\phi^{**}\otimes\id)(x) (1\otimes\theta)
&= (T\otimes 1)^* x_{13} (S\otimes \theta)
= (T\otimes 1)^* (x(1\otimes\theta))_{13} (S\otimes 1) \\
&\in (T\otimes 1)^* \big( A\otimes 1 \otimes\mc B_0(K) \big) (S\otimes 1)
\subseteq T^*(A\otimes 1)S \otimes \mc B_0(K).
\end{align*}
However, as $\phi(a) = T^*(a\otimes 1)S \in B$ for all $a\in A$, it follows
that $(\phi^{**}\otimes\id)(x) (1\otimes\theta) \in B\otimes\mc B_0(K)$
as required.  Similarly, $(1\otimes\theta) (\phi^{**}\otimes\id)(x)
\in B\otimes\mc B_0(K)$.

Now let $\theta\in\mc B_0(H')$ and consider
\[ (\phi^{**}\otimes\id)(x) (\theta\otimes 1)
= (T\otimes 1)^* x_{13} (S\theta\otimes 1). \]
Notice that $S\theta \in \mc B_0(H',H\otimes L')$.  To simplify the proof, notice
that we are always free to replace $L'$ by $L'\otimes\ell^2(X)$ for any index
set $X$, if we also replace $T$ by $T\otimes 1_{\ell^2(X)}$ and similarly for $S$.
That is, we are free to assume that there is some isometry $V:H'\rightarrow
H\otimes L'$.  Let $R = S\theta V^* \in \mc B_0(H\otimes L')$ so that
$S\theta = RV$ and hence
\[ (\phi^{**}\otimes\id)(x) (\theta\otimes 1)
= (T\otimes 1)^* x_{13} (RV\otimes 1)
= (T\otimes 1)^* x_{13}(R\otimes 1) (V\otimes 1). \]
For $\theta_1 \in \mc B_0(H), \theta_2\in\mc B_0(L')$ we have that
\[ x_{13}(\theta_1\otimes\theta_2\otimes 1)
= (x(\theta_1\otimes 1))_{13} (1\otimes\theta_2\otimes 1)
\in \mc B_0(H) \otimes \mc B_0(L') \otimes B, \]
as $x\in A\scten B$.  As $\mc B_0(H\otimes L') \cong \mc B_0(H)\otimes\mc B_0(L')$
it follows that $x_{13}(R\otimes 1) \in \mc B_0(H) \otimes \mc B_0(L') \otimes B$
and so
\begin{align*} (\phi^{**}\otimes\id)(x) (\theta\otimes 1)
\in (T\otimes 1)^* \big(\mc B_0(H) \otimes \mc B_0(L') \otimes B\big) (V\otimes 1)
\subseteq \mc B_0(H') \otimes B,
\end{align*}
as required.  Similarly $(\theta\otimes 1) (\phi^{**}\otimes\id)(x) 
\in \mc B_0(H') \otimes B$ and so $(\phi^{**}\otimes\id)(x) \in B\scten C$
as claimed.
\end{proof}

While we stated this result only in the ``one-sided'' case, it obviously
holds for maps of the form $\phi_1 \scten \phi_2$.

\section{Weakly almost periodic functionals}\label{sec:wap}

We now come to our principle application, that of studying ``weakly almost
periodic functionals'' on locally compact quantum groups, or more generally
Hopf von Neumann algebras.  Let $(\M,\Delta)$ be a Hopf von Neumann algebra,
so $\M$ is a von Neumann algebra and $\Delta:\M\rightarrow \M\vnten \M$ is a normal
unital injective $*$-homomorphism, coassociative in the sense that
$(\Delta\otimes\id)\Delta = (\id\otimes\Delta)\Delta$.  Then the preadjoint
of $\Delta$, say $\Delta_*:\M_*\proten \M_*\rightarrow \M_*$, turns $\M_*$ into
a completely contractive Banach algebra.  We shall write $\star$ for the product in
$\M_*$ and for the module action of $\M_*$ on $\M$ (and denote the module action of
$\M$ on $\M_*$ simply by juxtaposition).

We can thus import the normal Banach algebraic definition: $\wap(\M_*)$ consists
of those $x\in \M$ such that the orbit map $\M_*\rightarrow \M; \omega\mapsto
\omega\star x = (\id\otimes\omega)\Delta(x)$ is weakly compact.  Equivalently,
$x\in\wap(\M_*)$ if and only if $\Delta(x) \in \wap(\M\vnten \M)$, using the
notation we introduced in Section~\ref{sec:for_vn_algs}.

\begin{theorem}\label{thm:defn_wap_hvna}
Let $\wap(\M,\Delta)$ be the collection of those $x\in\wap(\M_*)$ such that
$x^*x, xx^* \in \wap(\M_*)$.  Then $\wap(\M,\Delta)$ is a unital $C^*$-subalgebra of
$\M$, and any $*$-subalgebra of $\wap(\M_*)$ is contained in $\wap(\M,\Delta)$.
\end{theorem}
\begin{proof}
As $\Delta$ is a $*$-homomorphism, $x\in\wap(\M,\Delta)$ if and only if
$\Delta(x) \in \wap(\M\vnten \M)$.  The result now follows from
Theorem~\ref{thm:sc_into_wap}.  As $\Delta(1)=1\otimes 1$ clearly
$\wap(\M,\Delta)$ is unital.
\end{proof}

Notice that, by definition, $x\in\wap(\M,\Delta)$ if and only if
$\Delta(x) \in \M\scten \M \subseteq \M\vnten \M$.

Let us recall some of the theory of Arens products on Banach algebras, for example
see \cite[Section~3]{dawsdis}, \cite[Section~2]{dawsdba}, and references therein.
Let $A$ be a Banach algebra, $X\subseteq A^*$ a closed $A$-submodule, and
suppose that $X$ is ``introverted'', meaning that if we define
\[ \ip{\Phi\cdot \mu}{a} = \ip{\Phi}{\mu\cdot a}, \quad
\ip{\mu\cdot\Phi}{a} = \ip{\Phi}{a\cdot\mu}
\qquad (a\in A,\mu\in A^*,\Phi\in A^{**}) \]
then $\Phi\cdot \mu, \mu\cdot\Phi\in X$ for all $\mu\in X, \Phi\in A^{**}$.
In this case, we can define products (the first and second Arens products)
on $X^*$ by
\[ \ip{\Phi\aone\Psi}{\mu} = \ip{\Phi}{\Psi\cdot\mu},\quad
\ip{\Phi\atwo\Psi}{\mu} = \ip{\Psi}{\mu\cdot\Phi}
\qquad (\Phi,\Psi\in X^*,\mu\in X). \]
If $X\subseteq\wap(A)$ then automatically $X$ is introverted, \cite[Lemma~1.2]{ll}.
In fact, for either $\aone$ or $\atwo$, $X^*$ becomes a ``dual Banach algebra''
(that is, the product is separately weak$^*$-continuous) if and only if
$X\subseteq\wap(A)$, see \cite[Proposition~2.4]{dawsdba}.

When $A=\M_*$ we have, using the notation of  Section~\ref{sec:for_vn_algs}, that
\[ \ip{\mu\aone\lambda}{x} = \ip{\mu\otimes_\aone\lambda}{\Delta(x)},
\quad \ip{\mu\atwo\lambda}{x} = \ip{\mu\otimes_\atwo\lambda}{\Delta(x)}
\qquad (x\in X\subseteq \M, \mu,\lambda\in X^*). \]
We remark that, by the Hahn-Banach Theorem, it does not matter if we work
with $X^*$ or $\M^*$.  Then Lemma~\ref{lem:wap_char} immediately shows that
if $X\subseteq\wap(\M_*)$, then $\aone=\atwo$ on $X^*$.

\begin{theorem}\label{thm:wap_submod}
For any Hopf von Neumann algebra, $\wap(\M,\Delta)$ is an $\M_*$-submodule of $\M$.
As such, $\wap(\M,\Delta)^*$ becomes a dual Banach algebra for either Arens
product (which agree).
\end{theorem}
\begin{proof}
Let $x\in\wap(\M,\Delta)$, so by definition, $\Delta(x)\in \M\scten \M
\subseteq \M\vnten \M$.  To be careful, let $y \in \M\scten \M$ be the image of
$\Delta(x)$.  Let $\omega\in \M_*$ and consider
\begin{align*} \Delta\big( \omega\star x \big)
&= \Delta\big( (\id\otimes\omega)\Delta(x) \big)
= (\id\otimes\id\otimes\omega)(\Delta\otimes\id)\Delta(x) \\
&= (\id\otimes\id\otimes\omega)(\id\otimes\Delta)\Delta(x)
= (\id\otimes\phi)\Delta(x),
\end{align*}
where $\phi:\M\rightarrow \M$ is the (normal) completely bounded map
$z \mapsto (\id\otimes\omega)\Delta(z)$.  Let $y' = (\id\scten\phi)(y)
\in \M\scten \M$ thanks to Theorem~\ref{thm:morphisms}.  For $\omega_1,\omega_2
\in \M_*$ and with $\kappa = \kappa_{\M_*}:\M_*\rightarrow \M^*$, we have that
\begin{align*} \ip{y'}{\kappa(\omega_1)\otimes\kappa(\omega_2)}
&= \ip{y}{\kappa(\omega_1)\otimes\phi^*\kappa(\omega_2)}
= \ip{y}{\kappa(\omega_1)\otimes\kappa\phi_*(\omega_2)} \\
&= \ip{\Delta(x)}{\omega_1 \otimes \phi_*(\omega_2)}.
\end{align*}
Here we used the embedding of $\M\scten \M$ into $\M\vnten \M$, the definition
of $\id\scten\phi$ and that $\phi$ is normal with preadjoint
$\phi_*:\M_*\rightarrow \M_*; \omega_2 \mapsto \omega_2\star\omega$.  Hence
\[ \ip{y'}{\kappa(\omega_1)\otimes\kappa(\omega_2)}
= \ip{x}{\omega_1 \star (\omega_2\star\omega)}
= \ip{\Delta(\omega\star x)}{\omega_1\otimes\omega_2}. \]
Thus the image of $y'\in \M\scten \M$ in $\M\vnten \M$ is simply
$\Delta(\omega\star x)$ and so $\omega\star x \in \wap(\M,\Delta)$.

Analogously, to show that $x\star\omega\in\wap(\M,\Delta)$ we show that
$\Delta(x\star\omega)\in \M\scten \M$.  As $\Delta(x\star\omega)
= (\phi'\otimes\id)\Delta(x)$ where $\phi'(z) = (\omega\otimes\id)\Delta(z)$
this will follow in the same way.
\end{proof}

\subsection{In the language of compactifications}

When $G$ is a locally compact group, $\wap(G) = \wap(L^1(G)) \subseteq L^\infty(G)$
is a commutative $C^*$-algebra with character space $G^{\wap}$ which becomes
a compact semitopological semigroup.  In fact, $G^{\wap}$ is ``maximal'' in the
sense that if $S$ is a compact semitopological semigroup and $\phi:G\rightarrow S$
a continuous (semi)group homomorphism, then there is a semigroup homomorphism
$\phi_0:G^{\wap}\rightarrow S$ factoring $\phi$.

We can turn this into a statement about algebras and coproducts in the usual
way (compare \cite{daws, soltan}).  However, in this setting, we would need a 
good notion of
a ``non-commutative'' or ``quantum'' semitopological semigroup.  The following is
now an obvious, but tentative, definition.

\begin{definition}
A compact quantum semitopological semigroup is a pair $(A,\Delta_A)$ where
$A$ is a unital $C^*$-algebra and $\Delta_A:A\rightarrow A\scten A$ is a
$*$-homomorphism, ``coassociative'' in the sense that the induced product on
$A^*$ is associative.
\end{definition}

As $A\scten A \subseteq A^{**}\vnten A^{**}$ by definition, the product on $A^*$ is
simply
\[ \ip{\mu\star\lambda}{a} = \ip{\Delta_A(a)}{\mu\otimes\lambda}
\qquad (a\in A, \mu,\lambda\in A^*). \]
The ``$C^*$-Eberlein algebras'' explored in \cite{dawsdas} fit into this framework,
thanks to \cite[Definition~3.6]{dawsdas} and \cite[Section~3.3]{dawsdas}.

\begin{theorem}\label{thm:wap_cqss}
Let $(\M,\Delta)$ be a Hopf von Neumann algebra and let $\wap = \wap(\M,\Delta)$
be as in Theorem~\ref{thm:defn_wap_hvna}.  Viewing $\wap\scten\wap$ as a subspace
of $\M\scten \M$, which in turn is a subspace of $\M\vnten \M$, we have that
$\Delta$ restricts to a map $\Delta_{\wap} : \wap \rightarrow \wap \scten \wap$.
\end{theorem}
\begin{proof}
Let $x\in\wap$ and let $y$ be the image of $\Delta(x)$ in $\M\scten \M$.
By Theorem~\ref{thm:inclusion} we need to show that $y\in SC(\wap\times\wap)$,
that is, that $(\mu\otimes\id)y, (\id\otimes\mu)y \in \wap$ for all $\mu\in \M^*$.

Let $\mu\in \M^*$ and choose a bounded net $(\omega_\alpha)$ in $\M_*$ converging
weak$^*$ to $\mu$.  For $\omega\in \M_*$ we have that
\[ \ip{(\mu\otimes\id)y}{\omega} = \ip{\mu}{(\id\otimes\omega)\Delta(x)}
= \lim_\alpha \ip{x}{\omega_\alpha\star\omega}
= \lim_\alpha \ip{x\star\omega_\alpha}{\omega}. \]
As $x\in\wap(\M_*)$ the map $\M_*\rightarrow \M; \tau\mapsto x\star\tau$ is
weakly compact, and so we may assume that $(x\star\omega_\alpha)$ converges
weakly.  By Theorem~\ref{thm:wap_submod}, this net is contained in $\wap$
which is a norm closed subspace, hence weakly closed.  We conclude that
$(x\star\omega_\alpha)$ converges to a member of $\wap$ and hence
$(\mu\otimes\id)y\in\wap$.  Analogously, $(\id\otimes\mu)y\in\wap$, as required.
\end{proof}

Combining this result with Theorem~\ref{thm:wap_submod} we see that
$(\wap,\Delta_{\wap})$ is a compact quantum semitopological semigroup.

We can now show that $\wap(\M,\Delta)$ has the required universal property to
be a ``compactification''.  Given $(A,\Delta_A)$ a compact quantum semitopological
semigroup, let $\theta:A\rightarrow \M$ be a $*$-homomorphism, and let $\tilde\theta:
A^{**}\rightarrow \M$ be the normal extension.  As $\Delta_A$ maps into $A\scten A
\subseteq A^{**}\vnten A^{**}$, the map $(\tilde\theta\otimes\tilde\theta)\Delta_A:
A\rightarrow \M\vnten \M$ makes sense.  If $(\tilde\theta\otimes\tilde\theta)\Delta_A
= \Delta\theta$ then we shall say that $\theta$ is a \emph{morphism}.  This is
equivalent to the restriction of $\theta^*$ to $\M_*$ being a Banach algebra
homomorphism $\M_*\rightarrow A^*$.

\begin{theorem}\label{thm:cqss_wap_compact}
Let $(\M,\Delta)$ be a Hopf von Neumann algebra.
Let $(A,\Delta_A)$ be a compact quantum semitopological semigroup and let
$\theta:A\rightarrow \M$ be a morphism.  Then $\theta(A) \subseteq \wap(\M,\Delta)$,
and $\wap(\M,\Delta)$ is the union of the images of all such $\theta$.
Furthermore, there is a $*$-homomorphism, intertwining the coproducts,
$\theta_0:A\rightarrow\wap$ which factors $\theta$.
\end{theorem}
\begin{proof}
Notice that we simply define $\theta_0$ to be the corestriction of $\theta$,
assuming that $\theta$ does map into $\wap$, and that as $(\wap,\Delta_{\wap})$
is itself a compact quantum semitopological semigroup, the inclusion map shows that
$\wap$ is the union of images of suitable $\theta$.

So it remains to show that for $a\in A$, we do have that $\theta(a)\in\wap$,
that is, that $\Delta(\theta(a)) \in \M\scten \M$.  However, $\Delta(\theta(a))
= (\tilde\theta\otimes\tilde\theta)\Delta_A(a)$.  By Theorem~\ref{thm:morphisms}
we have that $x = (\theta\scten\theta)\Delta_A(a) \in \M\scten \M$.  Let $y$ be
the image of $x$ in $\M\vnten \M$.  Let $\kappa=\kappa_{\M_*}:\M_*\rightarrow \M^*$
and recall that actually $\tilde\theta = (\theta^*\kappa)^* = \kappa^*\theta^{**}$.
Then, by definition of the various maps,
\begin{align*}
\ip{y}{\omega_1\otimes\omega_2} &=
\ip{x}{\kappa(\omega_1)\otimes\kappa(\omega_2)}
= \ip{(\theta^{**}\otimes\theta^{**})\Delta_A(a)}{\kappa(\omega_1)
   \otimes\kappa(\omega_2)}
= \ip{(\tilde\theta\otimes\tilde\theta)\Delta_A(a)}{\omega_1\otimes\omega_2}.
\end{align*}
Thus, as required, $\Delta(\theta(a))$ is in the image of $\M\scten \M$ in
$\M\vnten \M$.
\end{proof}

\section{Continuous analogues}\label{sec:cty}

So far we have worked with Hopf von Neumann algebras, non-commutative generalisations
of measure spaces.  By analogy, there should be ``continuous'' version of the
theory, namely one which works with $C^*$-bialgebras.  Recall that a $C^*$-bialgebra
is a pair $(A,\Delta_A)$ where $A$ is a $C^*$-algebra and $\Delta_A:A\rightarrow
M(A\otimes A)$ is a non-degenerate $*$-homomorphism which is coassociative.

In this section, we wish to treat abstract $C^*$-bialgebras, but also those
which arise from locally compact quantum groups, where we have more structure,
and in particular good interaction with the Hopf von Neumann theory.  We hence
proceed with a little generality.

Fix a $C^*$-bialgebra $(A,\Delta_A)$.  Let $(\M,\Delta)$ be a Hopf von Neumann
algebra and suppose we have an injective $*$-homomorphism $\theta:A\rightarrow \M$
which is non-degenerate in the sense that if $(e_\alpha)$ is a bounded approximate
identity for $A$ then $(\theta(e_\alpha))$ converges weak$^*$ to $1$ in $\M$.
Then $\theta$ extends to $\tilde\theta:M(A)\rightarrow \M$ which is also injective,
identifying $M(A)$ with $\{ x\in \M : x\theta(a), \theta(a)x \in \theta(A)
\ (a\in A) \}$.  We also denote by $\tilde\theta$ the normal extension $A^{**}
\rightarrow \M$.  These maps are compatible in the sense that if we view $M(A)$
as being $\{ x\in A^{**} : xA,Ax\subseteq A\}$ then $\tilde\theta$ restricted to
$M(A)$ agrees with the extension of $\theta$ from $A$ to $M(A)$.

We shall then make the further assumption that $(\tilde\theta\otimes\tilde\theta)
\Delta_A(a) = \Delta_\M(\theta(a))$ for all $a\in A$; this implies the same for
all $a\in M(A)$.  Again, we can either interpret this formula as meaning
\[ (\theta\otimes\theta)\big( \Delta_A(a)(b\otimes c) \big)
= \Delta_\M(\theta(a)) (\theta(b)\otimes\theta(c))
\qquad (a,b,c\in A), \]
or in terms of extensions to biduals, that is, including $M(A\otimes A)$ into
$A^{**}\vnten A^{**}$.

\begin{itemize}
\item
If $(A,\Delta_A)$ is an abstract $C^*$-bialgebra, then, for example, we may take
$\M=A^{**}$, and then form $\Delta_\M:A^{**}\rightarrow A^{**}\vnten A^{**}$ by
first considering $\Delta_A:A\rightarrow M(A\otimes A)\subseteq A^{**}\vnten A^{**}$
and then forming the normal extension.  Then $\theta=\kappa_A:A\rightarrow A^{**}$ is
the canonical map.
\item
If $A=C_0(\G)$ arises from a locally compact quantum group, then the most natural
choice is to take $\M=L^\infty(\G)$ with its usual coproduct.  Then $\theta$ is
the inclusion.
\end{itemize}

Notice that if $\M,\Delta_\M,\theta$ is any choice, then we always have a quotient
map $\phi:A^{**}\rightarrow \M$ such that $\phi\kappa_A = \theta$.  Indeed,
if $\M\subseteq\mc B(H)$ then $\phi$ will map onto $\theta(A)'' = \M$ by our
assumption on $\theta$, compare \cite[Section~2, Chapter~III]{tak}.  Then $\phi$
will intertwine the coproducts $\Delta_{A^{**}}$ and $\Delta_\M$.  We shall verify,
as we go along, that whether we work in $A^{**}$ or in $\M$ is unimportant.

As motivation for the following, consider a locally compact group $G$ and set
$A=C_0(G), \M=L^\infty(G)$ with $\theta$ the inclusion.  We wish to know when
$f\in C_b(G) = M(C_0(G))$ is in $\wap(G)$.  One abstract approach would be to
try to embed $M(A\otimes A)$ into $M(A)^{**}\vnten M(A)^{**}$ (so as to ask
when we land in $M(A)\scten M(A)$).  However, the comment after
Lemma~\ref{lem:slice_only_ma} shows that this cannot work.  Instead, we map
the problem into $\M$, and work with $\wap(\M,\Delta)$.  Our task then is to show
that this is independent of the choice of $\M$ (which it is!)

\begin{lemma}\label{lem:ma_mstar_submod}
The image of $M(A)$ in $\M$ is an $\M_*$-submodule.
\end{lemma}
\begin{proof}
For $x\in M(A)$ and $\omega\in \M_*$ we will show that $\tilde\theta(x)\star
\omega \in M(A)\subseteq \M$.
As $\theta$ is non-degenerate, by Cohen-Factorisation
(see \cite[Section~11]{bd} or \cite[Proposition~A2]{mnw}, for example)
there exists $b\in A,
\omega'\in \M_*$ with $\omega = \theta(b)\omega'$.  Then, for $c\in A,
\omega''\in \M_*$,
\begin{align*} \ip{(\omega\otimes\id)\Delta_\M(\tilde\theta(x)) \theta(c)}{\omega''}
&= \ip{\tilde\theta(x)}{\theta(b)\omega' \star \theta(c)\omega''}
= \ip{b\theta^*(\omega') \star c\theta^*(\omega'')}{x} \\
&= \ip{\theta^*(\omega') \otimes \theta^*(\omega'')}{\tilde\Delta_A(x)(b\otimes c)}
\\ &= \ip{\theta(d)}{\omega''},
\end{align*}
where $d = (\theta^*(\omega')\otimes\id)(\tilde\Delta_A(x)(b\otimes c))
\in A$ as $\tilde\Delta_A(x)(b\otimes c) \in A\otimes A$.
Similar remarks apply to slicing on the other side.
\end{proof}

\begin{lemma}\label{lem:slice_only_ma}
Let $\mu\in M(A)^*$ and let $\mu_0\in \M^*$ be a Hahn-Banach extension
(that is, $\mu_0\circ\tilde\theta = \mu$).  For $x\in M(A)$, both
$(\id\otimes\mu_0)\Delta_\M(\tilde\theta(x))$ and
$(\mu_0\otimes\id)\Delta_\M(\tilde\theta(x))$ depend only on $\mu$.
\end{lemma}
\begin{proof}
Considering $(\id\otimes\mu_0)\Delta_\M(\tilde\theta(x))$, our claim will
follow if $(\omega\otimes\id)\Delta_\M(\tilde\theta(x))$ is a member of
$\tilde\theta(M(A))$ for each $\omega\in\M_*$.  However, this follows from
the previous lemma.
\end{proof}

Unfortunately, there is no good reason why
$(\mu\otimes\id)\Delta_\M(\tilde\theta(x))$ should be a member of
$\tilde\theta(M(A))$.  We instead look to work more directly with $\M$.

\begin{proposition}
Let $x\in M(A)$.  Then $\tilde\theta(x) \in \wap(\M_*)$ if and only if
$x \in \wap(A^*)$ where $A^*$ is considered as the predual of the
Hopf von Neumann algebra $(A^{**},\Delta_{A^{**}})$.
\end{proposition}
\begin{proof}
Let $T:A^*\rightarrow A^{**}$ be the map $T(\mu) = (\mu\otimes\id)
\Delta_{A^{**}}(x)$.  By Lemma~\ref{lem:ma_mstar_submod} applied with
$\M=A^{**}$, we see that $T$ maps into $M(A)\subseteq A^{**}$
and so $x\in\wap(A^*)$ if and only if $T$ is weakly compact, if and only if
the corestriction $T:A^*\rightarrow M(A)$ is weakly compact.  Similarly
let $S:\M_*\rightarrow \M$ be $S(\omega) = (\omega\otimes\id)\Delta_\M(\tilde\theta(x))$
so that $\tilde\theta(x)\in\wap(\M_*)$ if and only if $S$ is weakly compact.

With $\phi:A^{**}\rightarrow \M$ as above, we have the commutative diagram
\[ \xymatrix{ A^* \ar[r]^T & M(A) \ar@{^(->}[r] & A^{**} \ar[ld]^\phi \\
\M_* \ar[u]^{\theta^*|_{\M_*}} \ar[r]^S & \M } \]
Thus, if $T$ is weakly compact, then so is $S$.

Suppose now that $S$ is weakly compact.  Let $\mu\in A^*$ with $\|\mu\|< 1$,
so again we can find $a\in A, \mu'\in A^*$ with $\mu = a\mu'$ and $\|a\|\|\mu'\|<1$.
Choose a net $(\omega'_\alpha)$ in the unit ball of $\M_*$ with
$\theta^*(\omega'_\alpha)\rightarrow\mu'$ weak$^*$ in $A^*$.  
For each $\alpha$ set $\omega_\alpha = \theta(a)\omega'_\alpha$.
For $\omega = \theta(c)\omega'\in \M_*$, we have that
\begin{align*} \ip{\phi T(\mu)}{\omega} &= 
\ip{\theta\big( (\mu'\otimes\id)( \Delta_{A^{**}}(x)(a\otimes c) ) \big)}{\omega'} \\
&= \lim_\alpha \ip{(\theta\otimes\theta)(\tilde\Delta_A(x)(a\otimes c))}
   {\omega'_\alpha\otimes\omega'} \\
&= \lim_\alpha \ip{\Delta_\M(\tilde\theta(x))(\theta(a)\otimes\theta(c))}
   {\omega'_\alpha\otimes\omega'} \\
&= \lim_\alpha \ip{S(\omega_\alpha)}{\omega}.
\end{align*}
As $S$ is weakly compact, $X = \{ S(\tau) : \tau\in \M_*, \|\tau\|\leq 1\}$ is
relatively weakly compact in $\M$.  The above shows that $\phi T(\mu)$ is in
the weak$^*$ closure of $X$, but as $X$ is relatively weakly compact and convex,
this agrees with the norm closure of $X$, which is a weakly compact set.
We conclude that $\phi T$ maps the unit ball of $A^*$ into a relatively
weakly compact subset of $\M$, that is, $\phi T$ is weakly compact.  As $T$ actually
maps into $M(A)$ and $\phi$ restricted to $M(A)$ is an isometry, it follows that
$T$ is weakly compact, as required.
\end{proof}

\begin{definition}
Let $\wap(A,\Delta_A) = \{ x\in M(A) : \tilde\theta(x) \in \wap(\M,\Delta_\M) \}$,
a $C^*$-subalgebra of $M(A)$.
By the proposition, this space depends only on $(A,\Delta_A)$.
\end{definition}

The following is the analogue of Theorem~\ref{thm:wap_cqss}.  It can again
be shown that the construction is independent of the choice of $\M$.

\begin{theorem}
Let $\wap = \wap(A,\Delta_A)$.
For $x\in\wap$, we have that $\Delta_\M(\tilde\theta(x))$ is
in $\wap\scten\wap \subseteq M(A)\scten M(A) \subseteq \M\scten \M \subseteq \M\vnten \M$.
As such, $\Delta_\M$ restricts to a map $\Delta_{\wap}:\wap\rightarrow
\wap\scten\wap$.
\end{theorem}
\begin{proof}
Exactly as in the proof of Theorem~\ref{thm:wap_cqss}, this will follow if we
can show that $(\mu\otimes\id)\Delta_\M(\tilde\theta(x)) \in \wap$ for
$\mu\in \M^*$ (and analoguously for $\id\otimes\mu$).  By weak compactness, it
suffices to show this for $\mu\in \M_*$, that is, that $\wap\subseteq \M$ is
an $\M_*$-submodule.  However, $\wap = M(A) \cap \wap(\M,\Delta_\M)$ and we know that
$\wap(\M,\Delta_\M)$ is an $\M_*$-submodule, so the result follows from
Lemma~\ref{lem:ma_mstar_submod}.
\end{proof}

We could now continue to prove an analogue of Theorem~\ref{thm:cqss_wap_compact} in
this setting.  We leave the details to the reader.

\subsection{For locally compact quantum groups}

For $G$ a locally compact group, the classical theory tells us that
$\wap(L^\infty(G)) = \wap(C_b(G))$, see for example \cite{ulger}.
We now make some remarks in this direction
in the setting of locally compact quantum groups.

We recall the notion of a locally compact quantum group $\G$ being
\emph{coamenable}, \cite{bt}.  In our setting, the most useful equivalent
definition is that $\G$ is coamenable if and only if $L^1(\G)$ has a bounded
approximate identity.

\begin{theorem}\label{thm:coamen}
Let $\G$ be coamenable.  Then $\wap(L^\infty(\G),\Delta)$ is contained in
$M(C_0(\G))$ and so agrees with $\wap(C_0(\G),\Delta)$.
\end{theorem}
\begin{proof}
If $x\in\wap(L^\infty(\G),\Delta)$ then $x\in\wap(L^1(\G))$, and a bounded
approximate identity argument shows that $x$ is contained in the norm closure
of $\{ \omega\star x: \omega\in L^1(\G) \}$ (as in the classical case,
compare \cite{ulger}).  Then use that $\omega\star x\in M(C_0(\G))$ for any
$x\in L^\infty(\G), \omega\in L^1(\G)$, see \cite[Theorem~2.4]{runde2};
or see directly \cite[Remark~4.5]{runde2}.
\end{proof}

We remark that similarly \cite[Theorem~4.4]{runde2} immediately implies
that $C_0(\G) \subseteq \wap(C_0(\G),\Delta) \subseteq
\wap(L^\infty(\G),\Delta)$, for any $\G$.

\section{Questions for further study}\label{sec:open}

We wrote $\scten$ by analogy with the theory of tensor products (compare,
for example, the extended Haagerup tensor product, \cite{er2}).  It is easy
to see that $A\scten B$ is isometrically isomorphic to $B\scten A$.
Is $\scten$ ``associative''?  Firstly, $(A\scten B)\scten C \subseteq
(A\scten B)^{**} \vnten C^{**}$ and $A\scten (B\scten C) \subseteq
A^{**}\vnten (B\scten C)^{**}$, and we cannot directly compare these,
but we can embed both spaces into $A^{**}\vnten B^{**}\vnten C^{**}$ using
the ideas of Section~\ref{sec:for_vn_algs}.  However, we have been unable to
decide if the two embedded spaces agree.

Condition (\ref{thm:main:three}) of Theorem~\ref{thm:main} is stated in
terms of ``rows'' and ``columns'' of operator matrices.  Such notions are
prominent in the theory of operator spaces.  Is there perhaps a way these ideas
could be made to work profitably for general operator spaces, not just
$C^*$-algebras?

The theory as applied to locally compact quantum groups gives maybe
three main questions:
\begin{itemize}
\item Is Theorem~\ref{thm:coamen} true without the coamenable hypothesis?
\item Does $\wap(L^\infty(\G),\Delta)$ always (or sometimes!) have an invariant
mean?
\item When does $\wap(L^\infty(\G),\Delta) = \wap(L^1(\G))$?
For an amenable discrete group $G$, this is true for $\widehat G$, that is,
$\wap(A(G)) = \wap(VN(G),\Delta)$.  This follows from
\cite[Proposition~2]{gran}, showing that $\wap(A(G)) = UCB(\widehat G)$,
and \cite[Proposition~2]{granier2}, which shows that $UCB(\widehat G)$ is
a $C^*$-algebra.  As far as we are aware, our question is open for
$\widehat{\mathbb F_2}$, for example.
\end{itemize}
Finally, we studied compactifications of $C^*$-Eberlein algebras in
\cite{dawsdas}: these are generated by coefficients of certain special
unitary corepresentations of $\G$.  Is it true that the coefficients of
any unitary corepresentation of $\G$ live in $\wap(C_0(\G),\Delta)$?

\vspace{5ex}

\noindent\emph{Author's Address:}
\parbox[t]{3in}{School of Mathematics\\
University of Leeds\\
Leeds\\
LS2 9JT}

\bigskip\noindent\emph{Email:} \texttt{matt.daws@cantab.net}


\begin{thebibliography}{99}

\bibitem{arens} R. Arens,
   ``The adjoint of a bilinear operation'',
   \emph{Proc. Amer. Math. Soc.} 2 (1951) 839--848.

\bibitem{bs} S. Baaj, G. Skandalis,
   ``Unitaires multiplicatifs et dualit\'e pour les produits
              crois\'es de {$C^*$}-alg\`ebres'',
   \emph{Ann. Sci. \'Ecole Norm. Sup. (4)} 26 (1993) 425--488.

\bibitem{bt} E. B{\'e}dos, L. Tuset,
   ``Amenability and co-amenability for locally compact quantum groups'',
   \emph{Internat. J. Math.} 14 (2003) 865--884.

\bibitem{bjm} J. Berglund, H. Junghenn, P. Milnes,
   ``Analysis on semigroups.
   Function spaces, compactifications, representations'',
   (John Wiley \& Sons, Inc., New York, 1989).

\bibitem{bd} F.\,F. Bonsall, J. Duncan,
   ``Complete normed algebras'',
   (Springer-Verlag, New York-Heidelberg, 1973).

\bibitem{conway} J.\,B. Conway,
   A course in functional analysis. Second edition.
   (Springer-Verlag, New York, 1990)

\bibitem{dawsdas} B. Das, M. Daws,
   ``Quantum Eberlein compactifications and invariant means'',
   to appear in \emph{Indiana Univ. Math. J.}, see arXiv:1406.1109v1 [math.FA]

\bibitem{daws} M. Daws,
  ``Remarks on the Quantum Bohr Compactification'',
  \emph{Illinois J. Math.} 57 (2013) 1131--1171.

\bibitem{dawswap} M. Daws,
   ``Characterising weakly almost periodic functionals on the measure algebra'',
   \emph{Studia Math.} 204 (2011) 213--234.

\bibitem{dawsdis} M. Daws,
   ``Multipliers, self-induced and dual Banach algebras'',
   \emph{Dissertationes Math.} 470 (2010) 62 pp.

\bibitem{dawsdba} M. Daws,
   ``Dual {B}anach algebras: representations and injectivity'',
   \emph{Studia Math.} 178 (2007) 231--275.

\bibitem{er2} E.\,G. Effros, Z.-J. Ruan,
   ``Operator space tensor products and {H}opf convolution algebras'',
   \emph{J. Operator Theory	} 50 (2003) 131--156.

\bibitem{ER} E.\,G. Effros, Z.-J. Ruan,
   Operator spaces.
   (Oxford University Press, New York, 2000).

\bibitem{eymard} P. Eymard,
   ``L'alg\`ebre de {F}ourier d'un groupe localement compact'',
   \emph{Bull. Soc. Math. France} 92 (1964) 181--236.

\bibitem{gran} E.\,E. Granirer,
   ``Weakly almost periodic and uniformly continuous functionals on
              the {F}ourier algebra of any locally compact group'',
   \emph{Trans. Amer. Math. Soc.} 189 (1974) 371--382.

\bibitem{granier2} E.\,E. Granirer,
   ``Density theorems for some linear subspaces and some
              {$C\sp*$}-subalgebras of {${\rm VN}(G)$}'',
   in Symposia {M}athematica, {V}ol. {XXII} ({C}onvegno
              sull'{A}nalisi {A}rmonica e {S}pazi di {F}unzioni su {G}ruppi
              {L}ocalmente {C}ompatti, {INDAM}, {R}ome, 1976)
   (1977) 61--70.

\bibitem{hnr} Z. Hu, M. Neufang, Z.-J. Ruan,
   ``Module maps over locally compact quantum groups'',
   \emph{Studia Math.} 211 (2012) 111--145.

\bibitem{kvvn} J. Kustermans, S. Vaes,
   ``Locally compact quantum groups in the von Neumann algebraic setting'',
   \emph{Math. Scand.} 92 (2003) 68--92. 

\bibitem{kv} J. Kustermans, S. Vaes,
   ``Locally compact quantum groups'',
   \emph{Ann. Sci. \'Ecole Norm. Sup. (4)} 33 (2000) 837--934.

\bibitem{ll} A.\,T.-M. Lau, R.\,J. Loy,
   ``Weak amenability of {B}anach algebras on locally compact groups'',
   \emph{J. Funct. Anal.} 145 (1997) 175--204.

\bibitem{mnw} T. Masuda, Y. Nakagami, S.\,L. Woronowicz,
   ``A {$C^\ast$}-algebraic framework for quantum groups'',
   \emph{Internat. J. Math.} 14 (2003) 903--1001.

\bibitem{meg} M. Megrelishvili,
   ``Fragmentability and representations of flows'',
   \emph{Topology Proc.} 2003 (27) 497--544.

\bibitem{paulsen} V.\,I. Paulsen,
   Completely bounded maps and operator algebras.
   (Cambridge University Press, Cambridge, 2002).

\bibitem{Pisier} G. Pisier,
   Introduction to operator space theory.
   (Cambridge University Press, Cambridge, 2003).

\bibitem{runde2} V. Runde,
   ``Uniform continuity over locally compact quantum groups'',
   \emph{J. Lond. Math. Soc. (2)} 80 (2009) 55--71.

\bibitem{runde} V. Runde,
   ``Connes-amenability and normal, virtual diagonals for measure
              algebras. {I}'',
   \emph{J. London Math. Soc. (2)} 67 (2003) 643--656.

\bibitem{salmi} P. Salmi,
   ``Quantum semigroup compactifications and uniform continuity on locally
   compact quantum groups'',
   \emph{Illinois J. Math.} 54 (2010) 469--483.

\bibitem{soltan} P. So{\l}tan,
   ``Quantum {B}ohr compactification'',
   \emph{Illinois J. Math.} 49 (2005) 1245--1270.

\bibitem{tak} M. Takesaki,
   Theory of operator algebras. I.
   (Springer-Verlag, Berlin, 2002).

\bibitem{ulger}  A. {\"U}lger,
   ``Continuity of weakly almost periodic functionals on {$L^1(G)$}'',
   \emph{Quart. J. Math. Oxford Ser. (2)} 148 (1986) 495--497.

\end{thebibliography}
\end{document}